\documentclass[12pt]{amsart}

\usepackage{amssymb,latexsym,amsmath,amsfonts,amsthm,amscd,amsxtra,enumerate}
\usepackage{mathtools,epsfig,hyperref}
\usepackage[all]{xy}
\usepackage[margin=1in]{geometry}
\usepackage[utf8]{inputenc}
\usepackage{graphicx}
\usepackage{slashbox}

\newtheorem{conj}{ Conjecture}[section]
\newtheorem{defn}[conj]{ Definition}

\newtheorem{example}[conj]{ Example}
\newtheorem{theorem}[conj]{ Theorem}

\newtheorem{remark}[conj]{ Remark}
\newtheorem{corollary}[conj]{ Corollary}
\newtheorem{proposition}[conj]{ Proposition}
\newtheorem{question}[conj]{ Question}

\providecommand\codim{\text{\rm codim}}
\providecommand\cmd{\text{\rm cmd}}
\providecommand\depth{\text{\rm depth}}
\providecommand\pdim{\text{\rm pdim}}
\providecommand\Tor{\text{\rm Tor}}

\providecommand{\BN}{{\mathbb N}}
\providecommand{\BZ}{{\mathbb Z}}
\providecommand{\BQ}{{\mathbb Q}}
\providecommand{\sk}{{\ensuremath{\sf k }}}

\begin{document}
\title{Modules With Pure Resolutions}
	
\author[H. Ananthnarayan]{H. Ananthnarayan}
\address{Department of Mathematics, I.I.T. Bombay, Powai, Mumbai 400076.}
\email{ananth@math.iitb.ac.in}
	
\author[Rajiv Kumar]{Rajiv Kumar}
\address{Department of Mathematics, I.I.T. Bombay, Powai, Mumbai 400076.}
\email{rajiv@math.iitb.ac.in}
	
	
	
\begin{abstract}
Let $R$ a standard graded algebra over a field $\sk$. In this paper, we show that the property of $R$ being Cohen-Macaulay is characterized by the existence of a pure Cohen-Macaulay $R$-module corresponding to any degree sequence of length at most $\depth(R)$. We also give a relation in terms of graded Betti numbers, called the Herzog-\ensuremath{\text{K}\ddot{\text{u}}\text{hl}} equations, for a pure $R$-module $M$ to satisfy the condition $\dim(R) - \depth(R) = \dim(M) - \depth(M)$. When $R$ is Cohen-Macaulay, we prove an analogous result characterizing all graded Cohen-Macaulay $R$-modules.
\end{abstract}

\maketitle

\section*{Introduction}
Let $S=\sk[X_1,\ldots,X_n]$ be a polynomial ring over a field $\sk$. M. Boij and J. \ensuremath{\text{S}\ddot{\text{o}}\text{derberg}} conjectured(\cite[Conjecture 2.4]{BS}) that given any degree sequence {\bf d} of length at most $n$, there is a pure Cohen-Macaulay $S$-module of type {\bf d}. This article is motivated by this conjecture, and the fact that it is true(see Remark \ref{purermk}(a)), as shown by D. Eisenbud and F. Schreyer. 

In this paper, we study modules of finite projective dimension, with pure resolutions, over an arbitrary standard graded ring $R$. The main tool used is the Cohen-Macaulay defect of an $R$-module $M$, denoted $\cmd(M)$. We show that the existence of a pure $R$-module $M$ of finite projective dimension forces $\cmd(R) \leq \cmd(M)$, a crucial observation that is used to prove the main results in this article. In particular, this shows that if such a module is Cohen-Macaulay, then the ring is also forced to be so. 

In \cite{BS}, Boij and \ensuremath{\text{S}\ddot{\text{o}}\text{derberg}} characterize graded Cohen-Macaulay modules over $S$, by relations among the graded Betti numbers of the module called the Herzog-\ensuremath{\text{K}\ddot{\text{u}}\text{hl}} equations. In Section \ref{HK}, we generalize this result to arbitrary Cohen-Macaulay standard graded algebras in Theorem \ref{thm3}. The Boij-Soderberg result extends a theorem of J. Herzog and M. \ensuremath{\text{K}\ddot{\text{u}}\text{hl}} from \cite{HK}, in which they prove it for pure Cohen-Macaulay modules over $S$. In Theorem \ref{thm2}, we generalize the result of Herzog and \ensuremath{\text{K}\ddot{\text{u}}\text{hl}} to a standard graded ring $R$, and characterize pure $R$-modules $M$ satisfying $\cmd(M) = \cmd(R)$.

The above results are used in Section \ref{App}, first to compute the multiplicity of such modules in Corollary \ref{cor1}, and then to get some interesting observations about cyclic modules with pure resolutions in Subsection \ref{Cyclic}. In Subsection \ref{poly}, we give a sufficient condition for a pure module over the polynomial ring to be Cohen-Macaulay. In particular, in Corollary \ref{cyclicapp}, we see conditions on the degree sequences that force a pure quotient of a polynomial ring to be Gorenstein, or a complete intersection. 

Finally, we see a characterization for $R$ to be Cohen-Macaulay in terms of pure modules. Using Proposition \ref{hklem2}, we prove (in Theorem \ref{thm1}) that $R$ is Cohen-Macaulay if and only if there is a pure Cohen-Macaulay $R$-module corresponding to any degree sequence of length at most $\depth(R)$.

Section \ref{Prelim} contains the notation, definitions, basic observations, and previous results that are needed in the rest of the article. 
	    
\section{Preliminaries}\label{Prelim}

\subsection{Notation}

\begin{enumerate}[{\rm a)}]
\item By $\sk$, we mean an infinite field, and $R$ denotes a standard graded $\sk$-algebra, i.e., $R_0 = \sk$, and $R$ is generated as an algebra over $\sk$ by finitely many elements of degree $1$. 

\item $M$ denotes a finitely generated graded $R$-module, $\mu(M)$, the {\it minimal number of generators} of $M$, $\cmd(M)=\dim(M)-\depth(M)$ is the {\it Cohen-Macaulay defect of $M$}, and the $i$th {\it total Betti number} of $M$ is $\beta_i^R(M) = \dim_{\sk}(\Tor_i^R(\sk,M))$. 
\item The $(i,j)$th-graded Betti number of $M$ is $\beta^R_{i,j}(M) = \dim_\sk(\Tor_i^R(M,\sk)_j)$. The {\it Betti diagram (or Betti table)} of $M$ is written as
\begin{center}
$\beta^R(M)$=
\begin{tabular}{|c |ccccc| }
\hline
\backslashbox{$j$}{$i$}& 0 & 1 & $\cdots$ & p&$\cdots$\\
\hline

$\vdots$& $\vdots$& $\vdots$& $\ddots$ & $\vdots$&$\ddots$\\
0 & $\beta_{0,0}$ & $\beta_{1,1}$ & $\cdots$ &$\beta_{p,p}$&$\vdots$\\
1 & $\beta_{0,1}$ & $\beta_{1,2}$ & $\cdots$ & $\beta_{p,p+1}$&$\vdots$\\
$\vdots$ & $\vdots$ &$\vdots$ & $\ddots$ &$\vdots$&$\ddots$\\
\hline

\end{tabular}
, i.e., the $(i,j)$th entry is $\beta_{i,i+j}^R(M).$
\end{center}

\item If $M_j$ is the $j$th graded component of $M$, the {\it shifted module} $M(d)$ is the $R$-module $M$ with graded components $M(d)_{j}=M_{j+d}$. The \emph{Hilbert series} of $M$ is $H_M(z)=\sum\limits_{j\in\BZ}\lambda(M_j)z^j$. 

\item By the Hilbert-Serre Theorem (see \cite[Section 4.1]{BH}), if $\dim(M) = r$, then there exists a unique Laurent polynomial $f(z)$ over $\BZ$ with $f(1) > 0$ such that $H_M(z)=\frac{f(z)}{(1-z)^r}$. The \emph{multiplicity} of $M$ is $e(M) = f(1)$.
\end{enumerate}
	    
\subsection{Pure Modules and Degree Sequences}    

\begin{defn}\label{puredef}
{\rm Let $R$ be standard graded, and $M$ a finitely generated graded $R$-module.
\begin{enumerate}[a)] 
\item An $R$-module $M$ has a {\it pure resolution} if each column of the Betti diagram of $M$ has at most one non-zero entry. If this happens, we say that $M$ is a {\it pure $R$-module}.
	    			
If $M$ has a pure resolution, and $\beta_{i,d_i} \neq 0$ for each $i < \pdim_R(M) + 1$, then $M$ is said to have a {\it pure resolution of type} ${\bf d} = (d_0,d_1,\ldots, d_p)$ if $M$ is a module of finite projective dimension $p$ over $R$, and of type $(d_0,d_1,\ldots)$ otherwise. 

\item For $p\leq \depth(R)$, a sequence ${\bf d} = (d_0,d_1,\ldots,d_p)$ is said to be a {\it degree sequence} of length $p$ if $d_i \in \BZ$, and $d_i < d_{i+1}$ for all $i$. 
\end{enumerate}
}\end{defn}

In the following remark, we record the relation between the Hilbert series, Betti numbers and shifts in the minimal free resolution of a pure module.

\begin{remark}[Hilbert series and Betti tables of pure modules]\hfill{}\\
{\rm Let $R$ be standard graded with Hilbert series $H_R(z) = \sum\limits _{i \geq 0}h_i^Rz^i$, $M$ a pure $R$-module of type ${\bf d} = (d_0,d_1,\ldots)$, with Hilbert series $H_M(z)  = \sum\limits_{i \in \BZ}h_i^Mz^i$, and $\beta_i = \beta_{i,d_i}^R(M)$. Then the Hilbert series and Betti table of $M$ can be obtained from one another as follows:
\begin{enumerate}[a)]
\item Fix $0 \leq i < \pdim_R(M) + 1$. If $K$ is the $i$th syzygy of $M$ in a minimal $R$-free resolution, then the exact sequence $$ 0 \rightarrow K \rightarrow R(-d_{i-1})^{\oplus \beta_{i-1}} \rightarrow \cdots \rightarrow R(-d_0)^{\oplus\beta_0}\rightarrow M\rightarrow 0$$ yields $H_K(z) = \left(\sum\limits_{k = 0}^{i-1} (-1)^{i-1-k} \beta_{k}z^{d_{k}}\right)H_R(z) + (-1)^i H_M(z)$.

Now, $d_i$ is the order of $H_K(z)$, and $\beta_i$ is the coefficient of $z^{d_i}$ in $H_K(z)$. Hence, we see that $d_i$, and $\beta_i$ can be inductively computed from $H_M(z)$ as 
$$d_{i}=\min\left\{j:\sum\limits_{k=0}^{i-1}(-1)^k\beta_kh^R_{j-d_k}-h^M_j\neq0\right\}, \quad\quad \beta_{i}= \left\vert\sum_{k=0}^{i-1}(-1)^k\beta_kh^R_{d_{i}-d_k}-h^M_{d_{i}}\right\vert.$$

\item The additivity of Hilbert series on exact sequences applied to a minimal $R$-resolution of $M$ shows that $h_i^M$ can be obtained from $d_i$, and $\beta_i$ by using $H_M(z) = H_R(z)\left(\sum\limits_{i\geq 0} (-1)^i \beta_i z^{d_i}\right)$.
\end{enumerate}
}\end{remark}

\begin{remark}[Pure modules over polynomial rings]\label{purermk}\hfill{}\\
{\rm Let $S = \sk[X_1,\ldots,X_n]$, and ${\bf d} = (d_0,d_1,\ldots,d_p)$, $p \leq n$, be a degree sequence.
\begin{enumerate}[a)] 
\item\label{ES1} (Eisenbud-Schreyer, \cite[Theorem 5.1]{ES})\hfill{}\\ There exists a pure Cohen-Macaulay $S$-module of type ${\bf d}$.

\item\label{HK1} (Herzog-K$\ddot{\text{u}}$hl, \cite[Theorem 1]{HK})\hfill{}\\ If $M$ is a pure $S$-module of type ${\bf d}$, then $M$ is Cohen-Macaulay if and only if $$\beta_i^S(M)=\beta_0^S(M)\prod\limits_{j\neq 0,i}\left|\frac{d_j-d_0}{d_j-d_i}\right|\quad \quad{\rm for}\ 1\leq i\leq p.$$ 

\item\label{purediag} We define two \emph{pure diagrams} corresponding to {\bf d}: $$\pi({\bf d}) = \frac{1}{\beta_0^S(M)}\beta^S(M), \text{ and }\pi'({\bf d}) = \frac{1}{\beta_p^S(M)}\beta^S(M),$$
where $M$ is a pure Cohen-Macaulay $S$-module of type {\bf d}, which exists by (\ref{ES1}). 

By (\ref{HK1}), these are independent of $S$ and $M$, and for $0 \leq i \leq p$, and $j \in \BZ$, we have 
$$\pi({\bf d})_{i,j} =\begin{cases}
\prod\limits_{k\neq i,0}\left|\frac{d_k-d_0}{d_k-d_i}\right| & \text{for }j=d_i\\
0 & \text{for } j\neq d_i
\end{cases}\quad \text{and}\quad
\pi'({\bf d})_{i,j} =\begin{cases}
\prod\limits_{k\neq i,p}\left|\frac{d_k-d_p}{d_k-d_i}\right| & \text{for }j=d_i\\
0 & \text{for } j\neq d_i
\end{cases}.$$

\item\label{ES2} Boij-S$\ddot{\text{o}}$derberg(\cite[Theorem 2]{BS2012}) and Eisenbud-Schreyer(\cite[Theorem 7.1]{ES}) prove:

Let $M$ be an $S$-module with $\codim(M) = c$ and $\pdim_S(M) = p$. There exist finitely many degree sequences ${\bf d}$ of length $s$, with $c\leq s\leq p$, such that $\beta^S(M)=\sum c_{\bf d}\pi({\bf d}),$ where $c_{\bf d} \in \BQ_+$.

Since $\pi({\bf d})$, and $\pi'({\bf d})$ are positive rational multiples of each other for each ${\bf d}$, we see that there exist $c'_{\bf d} \in \BQ_+$, such that $\beta^S(M)=\sum c'_{\bf d}\pi'({\bf d})$.
\end{enumerate}
}\end{remark}

\section{Herzog-\ensuremath{\text{K}\ddot{\text{u}}\text{hl}} Equations}\label{HK}

\subsection{Herzog-\ensuremath{\text{K}\ddot{\text{u}}\text{hl}} Equations over Arbitrary Rings} The Herzog-\ensuremath{\text{K}\ddot{\text{u}}\text{hl}} equations are relations among the graded Betti numbers of a module over a polynomial ring. These were initially studied in \cite{HK}, and then in \cite{BS}, to characterize Cohen-Macaulay modules. We look at similar relations over all standard graded $\sk$-algebras in this section.  

\begin{defn}\label{HKdef}
Let $R$ be a standard graded $\sk$-algebra, and $M$ a graded $R$-module of finite projective dimension $p$. Let $\beta_{i,j} = \beta_{i,j}^R(M)$. We say that $M$ satisfies the Herzog-\ensuremath{\text{K}\ddot{\text{u}}\text{hl}} equations if $$\sum\limits_{i,j}(-1)^ij^l\beta_{i,j}=0\quad {\rm for } \,\,l=0,\ldots,p-1.$$
\end{defn}

\begin{remark}\label{HKrmk}{\rm Let $S = \sk[X_1,\ldots,X_n]$, and $M$ be a graded $S$-module with $\pdim_S(M) = p$. \\
a) If $M$ is pure of type $(d_0,\ldots, d_p)$, then $M$ satisfies the Herzog-\ensuremath{\text{K}\ddot{\text{u}}\text{hl}} equations if and only if $\beta_i^S(M)=\beta_0^S(M)\prod\limits_{j\neq 0,i}\left|\frac{d_j-d_0}{d_j-d_i}\right|$ for $ 1\leq i\leq p.$ (For an explanation, see \cite{BS}, or the proof of the equivalence of (iii) and (iv) in Theorem \ref{thm2}).
In light of Remark \ref{purermk}(\ref{HK1}), the relations in Definition \ref{HKdef} are called the Herzog-\ensuremath{\text{K}\ddot{\text{u}}\text{hl}} equations. \\
b) In \cite[Section 2.1]{BS}, Boij and \ensuremath{\text{S}\ddot{\text{o}}\text{derberg}} extend the Herzog-\ensuremath{\text{K}\ddot{\text{u}}\text{hl}} result to $S$-modules that are not necessarily pure. They prove that an $S$-module $M$ satisfies the Herzog-\ensuremath{\text{K}\ddot{\text{u}}\text{hl}} equations in Definition \ref{HKdef} if and only if it is Cohen-Macaulay.
}\end{remark}

We begin with the following observation:

\begin{remark}\label{hklem}{\rm Let $h(z)=\sum\limits_{i,j}(-1)^i{b_{i,j}z^j}$ be a Laurent polynomial over $\BZ$. Then $(1-z)^n$ divides $h(z)$ if and only if the $b_{i,j}$'s satisfy the following equations:
$$\sum\limits_{i,j}(-1)^ij^lb_{i,j}=0\quad {\rm for } \,\,l=0,\ldots,n-1.$$
In order to see this,  let $a_l=\sum\limits_{i,j}(-1)^ij^lb_{i,j}=0$ for $l=0,\ldots,n-1.$ Then $h(1) = a_0$, and 
$$h^{(l)}(1)=a_l-\left(\sum\limits_{1 \leq k \leq l} k\right)a_{l-1}+\left(\sum\limits_{1\leq k_1 < k_2 \leq l} k_1k_2\right)a_{l-2} - \cdots+(-1)^{l-1}\left(\prod\limits_{k=1}^{l}k\right)a_1\text{ for }1 \leq l < n.$$ 
Thus $\begin{pmatrix} h(1)\\h^{(1)}(1) \\ \vdots \\ h^{(n-1)}(1) \end{pmatrix} = A \begin{pmatrix} a_0\\a_1 \\ \vdots \\ a_{n-1}\end{pmatrix}$, where $A$ is an $n \times n$ lower triangular matrix with diagonal entries $1$.

Thus we are done, since $(1-z)^n$ divides $h(z)$ if and only if $h^{(l)}(1)=0$ for $0 \leq l < n$, which, by the invertibility of $A$, happens if and only if $a_0, \ldots, a_{n-1}$ are all $0$.
}\end{remark}
	    
\begin{proposition}\label{hklem1} Let $R$ be standard graded, and $M$ a graded $R$-module of codimension $c$ and finite projective dimension $p$. If $\beta_{i,j} = \beta_{i,j}^R(M)$, and $h(z)=\sum\limits_{i = 0}^{p}\sum\limits_{j \in \BZ}(-1)^i{\beta_{i,j}z^j}$, then:
\begin{enumerate}[\rm a)]
\item $(1-z)^c$ divides $h(z)$, and $(1-z)^{c+1}$ does not divide $h(z)$,
\item the graded Betti numbers of $M$ satisfy:
$$\sum\limits_{i = 0}^p\sum\limits_{j \in \BZ}(-1)^ij^l\beta_{i,j}=0\quad {\rm for } \,\,l=0,\ldots,c-1, \text{ and }$$
\item the multiplicity of $M$ is $e(M)=\dfrac{e(R)}{c!}\left[\sum\limits_{i = 0}^{p}\sum\limits_{j \in \BZ}(-1)^{i+c}j^c\beta_{i,j}\right]$.
\end{enumerate}
\end{proposition}
	    
\begin{proof} The proof of (b) follows from (a) and the previous remark. Let $\dim(R) = d$. Hence, the Hilbert series of $R$ and $M$ are of the form $H_R(z)=\frac{f(z)}{(1-z)^d}$, and $H_M(z)=\frac{g(z)}{(1-z)^{d-c}}$, where $e(R) = f(1)\neq 0 \neq g(1) = e(M)$.

Since a minimal resolution of $M$ over $R$ is of the form $$0\rightarrow\bigoplus\limits_jR(-j)^{\beta_{p,j}}\rightarrow\cdots\rightarrow\bigoplus\limits_j R(-j)^{\beta_{1,j}}\rightarrow\bigoplus\limits_jR(-j)^{\beta_{0,j}}\rightarrow M\rightarrow0,$$ by the additivity of Hilbert series over exact sequences, we get 

\begin{equation*}\label{H(M)1} \frac{g(z)}{(1-z)^{d-c}} = H_M(z)=  H_R(z)  \left(\sum\limits_{i=0}^p\sum\limits_{j \in \BZ}(-1)^j\beta_{i,j}z^j\right) = \frac{f(z)h(z)}{(1-z)^d}.
\end{equation*}
Thus, $g(z)(1-z)^c = f(z)h(z)$.
Since $f(1)\neq0$, we see that $(1-z)^c$ divides $h(z)$, and $g(1) \neq  0$ implies that $(1-z)^{c+1}$ does not divide $h(z)$, proving (a), and hence (b).

In order to prove (c), let $h(z)=(1-z)^ch_1(z)$. Therefore, $h_1(1) = $
$$\frac{(-1)^c}{c!} h^{(c)}(1) =  \frac{(-1)^c}{c!} \left(\sum\limits_{i,j}(-1)^ij(j-1)(j-2)\cdots(j-c+1)\beta_{i,j}\right) = \frac{1}{c!}\sum\limits_{i,j}(-1)^{i+c}j^c\beta_{i,j},$$ where the last equality is by (b). Thus, $e(M) = g(1) = e(R) h_1(1)$, proving (c).
\end{proof}

\begin{remark}\label{hkrmk}{\rm Let $R$ be a standard graded $\sk$-algebra, and $M$ be an $R$-module of finite projective dimension $p$ and $\codim(M) = c$. By Remark \ref{hklem} and Proposition \ref{hklem1}, we see that $M$ satisfies the Herzog-\ensuremath{\text{K}\ddot{\text{u}}\text{hl}} equations if and only if $p \leq c$. By the Auslander-Buchsbaum formula, since $\pdim_R(M) < \infty$, this is equivalent to $\cmd(M) \leq \cmd(R)$.
}\end{remark}

Observe that if $R$ is Cohen-Macaulay in the above remark, i.e., $\cmd(R) = 0$, then $M$ satisfies the Herzog-\ensuremath{\text{K}\ddot{\text{u}}\text{hl}} equations if and only if $\cmd(M)= 0$, i.e., $M$ is Cohen-Macaulay. Thus we have proved: 
	    
\begin{theorem}\label{thm3} Let $R$ be a standard graded Cohen-Macaulay $\sk$-algebra and $M$ an $R$-module of finite projective dimension $p$. Then $M$ is Cohen-Macaulay if and only if its graded Betti numbers satisfy the equations
$$\sum\limits_{i=0}^p\sum\limits_{j \in \BZ}(-1)^ij^l\beta_{i,j}=0, \quad l=0,\ldots,p-1.$$
\end{theorem}

In particular, the above theorem extends the Boij-\ensuremath{\text{S}\ddot{\text{o}}\text{derberg}} result in Remark \ref{HKrmk}(b), characterizing graded Cohen-Macaulay $R$-modules by the Herzog-\ensuremath{\text{K}\ddot{\text{u}}\text{hl}} equations. 

\subsection{Herzog-\ensuremath{\text{K}\ddot{\text{u}}\text{hl}} Equations for Pure Modules}

We now show, in Theorem \ref{thm2}, that the Herzog-\ensuremath{\text{K}\ddot{\text{u}}\text{hl}} equations can be used to characterize pure modules $M$ of finite projective dimension over a standard graded $\sk$-algebra $R$, satisfying $\cmd(M) = \cmd(R)$. This generalizes the theorem of Herzog and \ensuremath{\text{K}\ddot{\text{u}}\text{hl}} in Remark \ref{purermk}(\ref{HK1}), and extends Theorem \ref{thm3} to pure modules over a standard graded $\sk$-algebra.

We now prove a key proposition, which is used in this section and the next.

\begin{proposition}\label{hklem2} Let $R$ be a standard graded $\sk$-algebra, and $M$ a pure $R$-module. Then $\codim(M) \leq \pdim_R(M)$. Moreover, if $\pdim_R(M)$ is finite, then $\cmd(R) \leq \cmd(M)$. 
\end{proposition}

\begin{proof} 
Note that if $\pdim_R(M)$ is infinite, there is nothing to prove. Hence, assume that $M$ is pure of type ${\bf d} = (d_0, \ldots, d_p)$, in particular, assume that $\pdim_R(M) = p$ is finite. With $\codim(M) = c$, we want to prove $c \leq p$.

If $p < c$, then, by Proposition \ref{hklem1}(b), $y_i = \beta_{i,d_i}$, $i = 0, \ldots, p$ is a non-trivial solution of the linear system of equations:
\begin{equation}\label{hk1}\sum\limits_i(-1)^id_i^ly_i=0, \quad  \,\, l=0,1,\ldots, c-1.
\end{equation}

The maximal minors of the coefficient matrix $A$ of this system are $(p+1) \times (p+1)$ Vandermonde matrices since $p+1 \leq c$. Thus, $d_i\neq d_j$ for all $i\neq j$ implies that $A$ has maximum rank $p+1$. This forces the system of equations (\ref{hk1}) to have only the trivial solution. This contradiction shows that $c \leq p$.

Finally, since $\pdim_R(M)$ is finite, then the Auslander-Buchsbaum formula implies that $\codim(M) \leq \pdim_R(M)$ is equivalent to $\cmd(R) \leq \cmd(M)$, proving the result.
\end{proof}

In the following example, we see that the inequality $\cmd(R) \leq \cmd(M)$ can be strict.
\begin{example}{\rm
Let $S = \sk[X_1,X_2]$, and $M = S/\langle X_1^2,X_1X_2 \rangle$. Then $M$ is pure of type ${\bf d} = (0,2,3)$.  In this case, $\cmd(M) = 1 > 0 = \cmd(S)$.

Note that $\beta^R(M)$ does not satisfy the Herzog-\ensuremath{\text{K}\ddot{\text{u}}\text{hl}} equations, since 
\[
\beta^R(M) =
\begin{tabular}{|c|ccc|}
\hline
\backslashbox{\emph j}{\emph i}&0&1&2\\
\hline
0&1&-&-\\
1&-&2&1\\
\hline
\end{tabular}
\quad \neq \quad
\begin{tabular}{|c|ccc|}
\hline
\backslashbox{\emph j}{\emph i}&0&1&2\\
\hline
0&1&-&-\\
1&-&3&2\\
\hline
\end{tabular}
= \beta_0^S(M)\pi({\bf d}).
\] 
}\end{example}

\begin{theorem}\label{thm2} Let $M$ be a pure module of type ${\bf d} = (d_0,\ldots,d_p)$ over a standard graded $\sk$-algebra $R$. Let $\pi({\bf d})$ and $\pi'({\bf d})$ be as in Remark \ref{purermk}(\ref{purediag}). Then the following are equivalent:\\
{\rm i)} $\cmd(M)=\cmd(R)$.\\
{\rm ii)} $\codim(M) = \pdim_R(M)$.\\
{\rm iii)} $\sum\limits_{i=0}^p(-1)^id_i^l\beta_{i,d_i}=0$ for $l=0,\ldots,p-1$, i.e. $M$ satisfies the Herzog-\ensuremath{\text{K}\ddot{\text{u}}\text{hl}} equations.\\
{\rm iv)} $\beta^R(M)=\beta_0^R(M)\pi({\bf d})$.\\ 
{\rm v)} $\beta^R(M)=\beta_p^R(M)\pi'({\bf d})$.
\end{theorem}
\begin{proof}
Since $M$ is a pure module of type ${\bf d}$, we have $\pdim_R(M) = p$ is finite. Hence, by the Auslander-Buchsbaum formula, (i) and (ii) are restatements of each other. 

Let $\codim(M) = c$. If (ii) is true, then Proposition \ref{hklem1}(b) shows that (iii) holds. In order to prove the converse, first note that by Remark \ref{hkrmk}, (iii) implies that $p \leq c$. Furthermore, since $M$ is a pure $R$-module, by Proposition \ref{hklem2}, we get $c \leq p$, proving (ii).

In order to see the equivalence of (iii) and (iv), note that the Herzog-\ensuremath{\text{K}\ddot{\text{u}}\text{hl}} equations are $p$ linear equations in $\beta_{0,d_0}, \ldots, \beta_{p,d_p}$, whose coefficient matrix is $
{\left[\begin{matrix} 1&\cdots&(-1)^{p}\\
d_0&\cdots&(-1)^{p}d_p\\
\vdots&\ddots&\vdots\\
d_0^{p-1}&\cdots&(-1)^{p}d_p^{p-1}
\end{matrix}\right].}$ Since $d_i \neq d_j$ for $i \neq j$, this a $p \times (p+1)$ matrix of full rank, and the solutions can be obtained by setting $\beta_{0,d_0}$ to be a free variable. Cramer's Rule proves the rest.

Similarly, setting $\beta_{p,d_p}$ to be a free variable proves the equivalence of (iii) and (v).
\end{proof}

\section{Applications}\label{App}
\subsection{Application to Multiplicity} In \cite{BS}, Boij and \ensuremath{\text{S}\ddot{\text{o}}\text{derberg}} compute the multiplicity of pure Cohen-Macaulay modules over polynomial rings. We use the above theorem to do the same for a pure module $M$ satisfying $\pdim_R(M) = \codim(M)$, over any standard graded $\sk$-algebra $R$.

\begin{corollary}\label{cor1}
Let $M$ be a pure module of type ${\bf d} = (d_0,\ldots,d_p)$ over a standard graded $\sk$-algebra $R$, with $\beta_0 = \beta_{0,d_0}^R(M)$. If $\cmd(R)=\cmd(M)$, then the multiplicity of $M$ is $$e(M)=e(R)\dfrac{\beta_0}{p!}\prod\limits_{j=1}^{p}(d_j-d_0).$$
\end{corollary}
\begin{proof}
Since $\cmd(R)=\cmd(M)$, and $\pdim_R(M) = p$ is finite, we get $\codim(M) = p$. Hence, by Theorem \ref{thm2} and Proposition \ref{hklem1}, we have 
\begin{eqnarray*}
e(M)&=& \beta_0\frac{e(R)}{p!}\sum\limits_{i=0}^p (-1)^{i+p}d_i^p\prod\limits_{j\neq i, j=1}^p\left\vert \frac{d_j-d_0}{d_j-d_i}\right\vert\\
&=& e(R)\frac{\beta_0}{p!}\prod\limits_{j=1}^{p}(d_j-d_0)\left(\sum\limits_{i=0}^p\frac{(-1)^{i+p}d_i^p}{\prod\limits_{j\neq i,j=0}^p\left\vert d_j-d_i\right\vert}\right).
\end{eqnarray*}

Let $V$  be the $(p+1) \times (p+1)$ Vandermonde matrix with $(i,j)$th entry $d_{i-1}^{j-1}$. We know that $\det(V) = \prod\limits_{0 \leq j < k \leq p}(d_k - d_j)$. On the other hand, expanding along the last row, we get $$\det(V) = \sum\limits_{i =0}^p (-1)^{i+p}d_i^p \left( \prod\limits_{0 \leq j < k \leq p,\ j \neq i \neq k}(d_k - d_j)\right).$$ Dividing the second expression for $\det(V)$ by the first, we see that $\sum\limits_{i=0}^p(-1)^{i+p}\frac{d_i^p}{\prod\limits_{j\neq i,j=0}^p\left\vert d_j-d_i\right\vert} = 1$. Substituting this in the expression for $e(M)$ proves the corollary.
\end{proof}

\subsection{Applications to Cyclic Modules}\label{Cyclic}
In this subsection, we study pure cyclic modules. We first look at Betti diagrams of pure cyclic Cohen-Macaulay modules of projective dimension 2.

\begin{theorem}\label{cyclic} Let $R$ be a standard graded $\sk$-algebra, and $I$ a homogeneous ideal in $R$, such that $\codim(R/I) = \pdim_R(R/I) = 2$. If $R/I$ is a pure $R$-module, then there exists a regular sequence  $g_1,g_2$ such that $\beta^R(R/I)=\beta^R(R/\langle g_1,g_2\rangle^i )$ for some $i$.
\end{theorem}
\begin{proof} Let $R/I$ be a pure $R$-module of type $(0,d_1,d_2)$. Since $\codim(R/I) = \pdim_R(R/I)$ and $\beta^R_0(R/I)=1$, by Theoqrem \ref{thm2}, we have  $\beta^R_1(R/I)=d_2/(d_2-d_1)$ and $\beta^R_2(R/I)=d_1/(d_2-d_1)$. Therefore $r = d_2-d_1 =$ gcd$(d_1,d_2)$. 
	    
Write $d_1=rd$ and let $g_1,g_2$ be an $R$-regular sequence of degree $r$. By using induction on $d$, it can be seen that $R/\langle g_1,g_2\rangle^d$ is a pure $R$-module of type $(0,rd,rd+r)$. Therefore  $R/\langle g_1,g_2\rangle^d$ and $R/I$ are pure of type $(0,rd,rd+r)$, with $\beta_0^R(R/I) = \beta_0^R(R/\langle g_1,g_2\rangle^d) = 1$. Hence, by Theorem \ref{thm2}, $\beta^R(R/\langle g_1,g_2\rangle^d)=\beta^R(R/I)$.
\end{proof}

In the above theorem, we have seen that $\beta^R(R/\langle g_1,g_2\rangle^d)=\beta^R(R/I)$. One can ask: Is $R/I\simeq R/\langle g_1,g_2\rangle^d$? The following example shows that it need not be true.
	  
\begin{example}
	{\rm Let $S=\sk[X_1,X_2,X_3]$, and $I=\langle X_1X_2, X_2X_3, X_1X_3\rangle$. Then $S/I$ is a pure $S$-module of type $(0,2,3)$, and $\beta^S(S/I) = \beta^S(S/\langle X_1,X_2 \rangle^2)$, illustrating Theorem \ref{cyclic}.
	
	We know that if $g_1,g_2$ is a regular sequence,  then $S/\langle g_1,g_2\rangle ^i$ is a Cohen-Macaulay for all $i\in\BN$. Since $S/I^2$ is not Cohen-Macaulay, there does not exist any regular sequence $g_1,g_2$ such that $I$ isomorphic to $\langle g_1,g_2\rangle ^i$.}
\end{example}
	  
The following example shows that Theorem \ref{cyclic} need not hold if $\pdim_R(R/I) \geq 3$.
\begin{example} {\rm  Let $S=\sk[X_1,X_2,X_3]$ and $I=\langle X_1X_2, X_2X_3, X_1X_3,X_1^2-X_2^2,X_1^2-X_3^2\rangle$. Then $S/I$ is a pure  Gorenstein quotient of $S$ of type $(0,2,3,5)$.

Note that $\beta^S(S/I) \neq \beta^S(S/\langle g_1,g_2,g_3\rangle^i)$, for any $i$, where $g_1$, $g_2$, $g_3$ is an $S$-regular sequence. This follows since $\beta_1^S(S/I)=5$, and for any regular sequence   $g_1,g_2,g_3$, we have  $\beta_1^S(S/\langle g_1,g_2,g_3\rangle)=3$, and $\beta_1^S(S/\langle g_1,g_2,g_3\rangle^i)\geq 6$ for $i\geq2$.
}\end{example}	  
	    	    
We have seen in Theorem \ref{cyclic} that $\beta^R(R/I)=\beta^R(R/\langle g_1,g_2\rangle^i)$. Since $R/\langle g_1,g_2\rangle^i$ is a pure module for each $i$, we have the following natural question.
	    
\begin{question} {\rm Let $R$ be a graded $\sk$-algebra and $I$ be a homogeneous ideal in $R$ such that $R/I$ is a pure $R$-module. If $\cmd(R/I)=\cmd(R)$, then is $R/I^i$ a pure $R$-module for every $i$?}
\end{question}

\subsection{Modules over Polynomial Rings}\label{poly}

We first give a sufficient condition for a pure module over a polynomial to be Cohen-Macaulay.  This leads to conditions on the degree sequences, which force pure cyclic quotients to be Gorenstein, or a complete intersection.	    

\begin{theorem}
Let $S$ be a polynomial ring, ${\bf d}=(d_0,d_1,\dots,d_p)$ be a degree sequence with $1 \leq \pi'({\bf d})_{0,d_0}$, and $M$ a pure $S$ module of type ${\bf d}$. If $\beta_{0,d_0}(M) \leq \beta_{p,d_p}(M)$, then $M$ is Cohen-Macaulay.
\end{theorem}
\begin{proof}
By Remark \ref{purermk}(\ref{ES2}), we can write 
$$\beta^S(M)=c_1\pi'({\bf d})+c_2\pi'({d_0, d_1,\dots,d_{p-1}})+\dots+ c_{p-1}\pi'({d_0, d_1})+ c_{p}\pi'({d_0}),$$ where $c_i \in \BQ_{\geq 0}$, with $c_1 > 0$. The above equation shows that $c_1=\beta^S_{p,d_p}(M)$. Now 
$$\beta^S_{0,d_0}(R)= c_1\pi'({\bf d})_{0,d_0}+c_2\pi'({d_0, d_1,\dots,d_{p-1}})_{0,d_0}+\dots+ c_{p-1}\pi'({d_0, d_1})_{0,d_0}+ c_{p}\pi'({d_0})_{0,d_0}$$ 
$$\geq c_1\pi'({\bf d})_{0,d_0} \geq c_1 =\beta^S_{p,d_p}(M) \geq \beta^S_{0,d_0}(M).$$ 
Thus $\pi'({\bf d})_{0,d_0} = 1$, $\beta^S_{0,d_0}(M) = \beta^S_{p,d_p}(M)$. Moreover, $\beta^S(M) = \beta^S_{p,d_p}(M)\pi'({\bf d})$, since $c_i = 0$ for $i \geq 2$. Therefore, $M$ is Cohen-Macaulay by Remark \ref{purermk}(\ref{HK1}) and (\ref{purediag}). 
\end{proof}

\begin{corollary}\label{cyclicapp}
Let $S$ be a polynomial ring, ${\bf d}=(0,d_1,\dots,d_p)$ be a degree sequence, and ${\bf e} = (e_1,\ldots, e_p)$, where $e_i = d_i - d_{i-1}$ for $i = 1,\ldots,p$. Let $I \subset S$ be homogeneous such that $R=S/I$ is a pure $S$-module of type ${\bf d}$. Then the following hold:
\begin{enumerate}[\rm a)]
\item  If $\pi'({\bf d})_{0,d_0} \geq 1$, then $R$ is Gorenstein.
\item If ${\bf e}$ is non-decreasing, then $e_1 = \cdots = e_p$, and $R$ is a complete intersection.
\end{enumerate}   
\end{corollary}
	    
\begin{proof} a) Note that $R$ is Gorenstein if and only if $R$ is Cohen-Macaulay, and $\beta^S_p(R) = 1$. Both the latter conditions hold by the previous theorem, hence $R$ is Gorenstein. 

\noindent
b) We have $d_p = \sum_{i  =1}^p e_i =   \sum_{i  = 1}^j e_i + \sum_{i  = j+1}^p e_i \geq \sum_{i  = 1}^j e_i + \sum_{i  = 1}^{p - j} e_i = d_j + d_{p-j},$ for each $j$, since ${\bf e}$ is non-decreasing. In particular, $\frac{d_p - d_j}{d_{p-j}} \geq 1$, with equality holding for each $j$ if and only if $e_1 = \cdots = e_p$. 

Now, by definition, $\pi'({\bf d})_{0,0} = \prod\limits_{j = 1}^{p-1} \left( \dfrac{d_p - d_j}{d_{j}} \right) = \prod\limits_{j = 1}^{p-1} \left( \dfrac{d_p - d_j}{d_{p-j}} \right) \geq 1$. Hence, by the previous theorem, $\pi'({\bf d})_{0,0} = 1$, which forces $\frac{d_p - d_j}{d_{p-j}} = 1$ for each $j$, and hence, we see that $e_1 = \cdots = e_p$.

Thus, if $d_1 = d$, then ${\bf d} = (0, d, 2d, \ldots, pd)$. We see that $I$ is generated by a regular sequence, since $\mu(I) = \beta_1^S(R) = \beta_0^S(R)\prod\limits_{k = 2}^p\left\vert\dfrac{d_k}{d_k - d_1}\right\vert = p = \codim(R)$. Thus $R$ is a complete intersection.
\end{proof}

\subsection{Characterization of Cohen-Macaulay Rings }\label{CM}

In Theorem \ref{thm1}, we give a characterization of Cohen-Macaulay graded $\sk$-algebras in terms of pure Cohen-Macaulay modules. We first prove the following proposition.	    

\begin{proposition}\label{prop}
Let $R$ be a standard graded Cohen-Macaulay $\sk$-algebra of dimension $n$, and ${\bf d} = (d_0,d_1,\ldots, d_p)$, $p \leq n$, be a degree sequence. Then there exists a pure Cohen-Macaulay $R$-module of type ${\bf d}$.
\end{proposition}

\begin{proof}
Let $\depth(R) = r$. Since $\sk$ is infinite, there exists a regular sequence $l_1,\ldots,l_r$ of linear forms in $R$. Since $R$ is Cohen-Macaulay, $R$ is free of finite rank over $S=\sk[l_1,\dots,l_r]$, which is isomorphic to a polynomial ring over $\sk$. 

By Remark \ref{purermk}(\ref{ES1}), there exists a Cohen-Macaulay $S$-module $M$, which has a pure resolution over $S$ of type ${\bf d}$. Then $M\otimes_SR$ is a pure Cohen-Macaulay $R$-module of type ${\bf d}$.
\end{proof}

The next theorem follows from Propositions \ref{hklem2} and \ref{prop}.

\begin{theorem}\label{thm1}
Let $R$ be a standard graded $\sk$-algebra. Then the following are equivalent:\\
{\rm i)} there exists a pure Cohen-Macaulay $R$-module of type ${\bf d}$, for every degree sequence ${\bf d}$ of length at most $\depth(R)$.\\
{\rm ii)} there exists a pure Cohen-Macaulay $R$-module of finite projective dimension.\\
{\rm iii)} $R$ is Cohen-Macaulay.
\end{theorem}	 

\begin{proof}
Statement (ii) follows immediately from (i). Since $M$ is Cohen-Macaulay if and only if $\cmd(M) = 0$, the proof of (ii) $\Rightarrow$ (iii) is a consequence of Proposition \ref{hklem2}, and (iii) $\Rightarrow$ (i) is the content of Proposition \ref{prop}.
\end{proof}

\end{document}